\documentclass[twoside]{amsart}
\usepackage{latexsym}
\usepackage{amssymb,amsmath,amsopn}
\usepackage[dvips]{graphicx}   
\usepackage{color,epsfig}      
               {\begin{list}{}{\leftmargin#1\rightmargin#2}\item{}}%
               {\end{list}}

\newtheorem{thm}{Theorem}

\newtheorem{lem}{Lemma}
\theoremstyle{definition}
\newtheorem{defn}{Definition}
\newtheorem{rem}{Remark}

\newtheorem{conj}{Conjecture}
\newtheorem{prob}{Problem}

\newtheorem{cor}{Corollary}
\newtheorem{prop}{Proposition}
\renewcommand{\Re}{\mathbb R}

\renewcommand{\S}{\mathbb{S}}
\newcommand{\HH}{\mathbb{H}}

\newcommand{\Z}{\mathbb Z}

\def\bea{\begin{eqnarray}}
\def\eea{\end{eqnarray}}

\DeclareMathOperator{\inter}{int}

\DeclareMathOperator{\conv}{conv}

\DeclareMathOperator{\area}{area}

\begin{document}

\title[On some average properties of convex mosaics]{On some average properties of convex mosaics}
\author[G. Domokos and Z. L\'angi] {G\'abor Domokos and Zsolt L\'angi}
\address{G\'abor Domokos, MTA-BME Morphodynamics Research Group and Dept. of Mechanics, Materials and Structures, Budapest University of Technology,
M\H uegyetem rakpart 1-3., Budapest, Hungary, 1111}
\email{domokos@iit.bme.hu}
\address{Zsolt L\'angi, MTA-BME Morphodynamics Research Group and Dept. of Geometry, Budapest University of Technology,
Egry J\'ozsef utca 1., Budapest, Hungary, 1111}
\email{zlangi@math.bme.hu}
\thanks{The authors are very grateful to Rolf Schneider for his repeated encouragement which contributed to shape this manuscript.
We also thank Egon Schulte  and Frank Morgan for their positive comments, and an anonymous referee for careful reading and valuable comments. Supported by the NKFIH Hungarian Research Fund grant 119245 and of grant BME FIKP-V\'IZ by EMMI is kindly acknowledged.
ZL has been supported by grant UNKP-19-4 New National Excellence Program of the Ministry of Innovation and Technology
and the J\'anos Bolyai Research Scholarship of the Hungarian Academy of Sciences.
}


\begin{abstract}
In a convex mosaic in $\Re ^d$ we denote the average number of vertices of a cell by $\bar v$ and
the average number of cells meeting at a node by $\bar n$. Except for the $d=2$ planar case, there is no known formula 
prohibiting points in any range of the $[\bar n, \bar v]$ plane (except for the unphysical $\bar n, \bar v < d+1$ strips).
Nevertheless,  in $d=3$ dimensions if we plot the 28 points corresponding to convex uniform
honeycombs, the 28 points corresponding to their duals and the 3 points corresponding to Poisson-Voronoi, Poisson-Delaunay
and random hyperplane mosaics, then these points appear to accumulate on a narrow strip of the $[\bar n, \bar v]$ plane. To explore this phenomenon
we introduce the harmonic degree $\bar h= \bar n\bar v/(\bar n + \bar v)$ of a $d$-dimensional mosaic. We show that the observed narrow strip
on the $[\bar n, \bar v]$ plane corresponds to a narrow range of $\bar h$. We prove that for every $\bar h^{\star} \in (d, 2^{d-1}]$ there exists a convex mosaic with harmonic degree $\bar h^{\star}$
and we conjecture that there exist no $d$-dimensional mosaic outside this range.
We also show that the harmonic degree has deeper geometric interpretations.
In particular, in case of Euclidean mosaics it is related to the average of the sum of vertex angles and their polars, and in case of 2D mosaics, it is related to the average excess angle. 
\end{abstract}
\maketitle

\tableofcontents

\section{Introduction}\label{sec:intro}

\subsection{Definition and brief history of mosaics}\label{subsec:history}
A $d$-dimensional \emph{mosaic} $\mathcal{M}$ is a  countable system of compact domains in $\Re ^d$, with nonempty interiors, that cover the whole space and have pairwise no common interior points \cite{Schneider}. 
We call a mosaic \emph{convex} if these domains are convex and in this case all domains are convex polytopes \cite[Lemma 10.1.1]{Schneider}. In this paper we deal only with convex mosaics. We call these polytopes the \emph{cells} of the mosaic, the $k$-dimensional faces of the cells, for $k=1,2,\ldots,d-1$, the \emph{faces} of the mosaic, and the vertices of the cells the \emph{nodes} of the mosaic. In particular, in case of $3$-dimensional mosaics, we may use the term face instead of facet of the mosaic.
A cell having $v$ vertices is called a \emph{cell of degree $v$}, and a node  which is the vertex of  $n$ cells is called a \emph{node of degree $n$}. Our prime focus is to determine how average values of these quantities, denoted by $\bar n$ and $\bar v$, respectively, 
depend on each other. 
We remark that for planar regular mosaics, the pair $\{ \bar v, \bar n \}$ is called the Schl\"afli symbol of the mosaic so, by generalizing this concept, we will refer to the $[\bar n, \bar v]$ plane as the \emph{symbolic plane} of convex mosaics.
These, and closely related global averages have been studied before and proved to be  powerful tools in the geometric study of mosaics: in \cite{morgan} the planar isoperimetric problem restricted to convex polygons with $v < 6$ vertices is resolved using these quantities.

Our main focus will be \emph{face-to-face} mosaics, in which any two distinct cells intersect in a common face or have empty intersection.
Unless stated otherwise, any mosaic discussed in our paper will be a convex face-to-face mosaic and we will only discuss non face-to-face mosaics in Subsection \ref{ss:nff}.
Furthermore, we assume that the mosaic is \emph{normal}, that is, for some $0 < r <R$ each cell contains a ball of radius $r$, and is contained in a ball of radius $R$ (see, e.g. \cite{Schulte2}).
This implies, in particular, that the volumes of the cells are bounded from above, and that the mosaic is \emph{locally finite}; that is, each point of space belongs to finitely many cells.
We note that a precise definition of $\bar{v}$ and $\bar{n}$ can be obtained in the usual way, that is, by taking the limit of the average degrees of cells/nodes contained in a large ball whose radius tends to infinity. Here, we always tacitly assume that these limits exist.

Geometric intuition suggests that $\bar v$ and $\bar n$ should have an inverse-type relationship: more polytopes meeting at a
node implies smaller internal angles in the polytopes, which, in turn, suggests a smaller number of vertices for each polytope. To be able to verify this intuition we introduce
\begin{defn}\label{harmonic}
The \emph{harmonic degree} of a mosaic $\mathcal{M}$ is defined as
\begin{equation}\label{q}
 \bar h (\mathcal{M}) = \frac{\bar n (\mathcal{M}) \bar v (\mathcal{M})}{\bar n (\mathcal{M}) + \bar v (\mathcal{M})},
\end{equation} 
where $\bar v (\mathcal{M}), \bar n (\mathcal{M})$ denote the average cell and nodal degrees of $\mathcal{M}$, respectively.
\end{defn}
The variation of the harmonic degree $\bar h$ (computed on an ensemble of mosaics)  may  
serve as a measure of how good our intuition was: a constant value of $\bar h$ describes an exact inverse-type relationship
while small variation of $\bar h$ still indicates that our intuitive approach is, to some extent, justified.
To describe a deeper, geometric meaning of the harmonic degree we introduce
\begin{defn}
Let $\mathcal{M} \in \Re ^d$ be a mosaic, $C$ be a cell of $\mathcal M$ and $p$ be a vertex of $C$. Then
the \emph{total angle}  $\Omega(C,p)$ of the pair $(C,p)$ is the sum of the internal and external angles of $C$ at $p$; the former defined as the  surface area of the spherical convex hull of the unit tangent vectors of the edges of $C$ at $p$, the latter defined as the surface area of the set of the outer unit normal vectors of $C$ at $p$. The \emph{average total angle} $\bar \Omega (\mathcal{M})$ associated with the mosaic $\mathcal{M}$ is defined as the average of  $\Omega(C,p)$, taken over all pairs $(C,p)$ in $\mathcal M$.
\end{defn}
Although $\bar h$ appears to be a combinatorial property and $\bar \Omega $ appears to be a metric property of the mosaic, nevertheless, they are closely linked, which is expressed in
\begin{thm}\label{thm:harmonic}
Let $\mathcal{M}$ be a convex, face-to-face mosaic in $\Re ^d$ and let $S_{d-1}$ denote the surface area of the $d$-dimensional unit sphere. Then we have
\[
\bar h (\mathcal{M}) \bar \Omega (\mathcal{M}) = S_{d-1}.
\]
\end{thm}
We will prove Theorem \ref{thm:harmonic} in Section \ref{s:p1},  and in Section \ref{qvary}  we extend it to $2$-dimensional spherical mosaics. Since there is no natural definition of average for hyperbolic mosaics (cf. also Subsection \ref{ss:noneupolar}), Theorem~\ref{thm:harmonic} cannot be extended to mosaics in hyperbolic planes.  
While $\Omega(\mathcal{M}) =\pi$ is constant in $d=2$ dimensions for Euclidean mosaics (implying, via Theorem \ref{thm:harmonic},  constant value for the harmonic degree $\bar h$)  however, if $d>2$ then $\Omega(\mathcal{M})$ may vary, so our original intuition appears to become ambiguous for $d>2$ and the variation of $\bar h$ will serve as an indicator of this ambiguity.

In one dimension ($d=1$) we have $v=n=2$ for each cell and vertex and thus, trivially $\bar h=1$ for all mosaics. In two dimensions one can have cells and nodes of various degrees, nonetheless,
it is known \cite[Theorem 10.1.6]{Schneider} that for all convex mosaics $\bar h=2$. 

The situation in $d=3$ dimensions appears, at least at first sight, to be radically different. Schneider and Weil \cite{Schneider} provide the general 
equations governing 3D random mosaics. In Section \ref{s:Schneidermatrix} we present an elementary proof that the same governing equations hold for any convex mosaic under some simple finiteness condition. These equations have three variable parameters. We also show that, beyond the trivial
inequalities $\bar v , \bar n \geq 4$ these formulae do not yield additional constraints on $\bar n, \bar v $ suggesting that in the $[\bar n, \bar v]$ symbolic plane, except for the unphysical domains characterized by $\bar n , \bar v <4$, we might expect to see mosaics \emph{anywhere}. However, this is not the case if we look at the best known mosaics: uniform honeycombs. The latter
are a special class of convex mosaics where cells  are uniform polyhedra
and all nodes are equivalent under the symmetry group of the mosaic. The list of all possible convex uniform honeycombs was completed only recently by Johnson \cite{Johnson} who
described 28 such mosaics (for more details on the 28 uniform honeycombs see \cite{Grunbaum, Deza} and more details on the history see \cite{Senechal}).
To provide the complete list of these 28 honeycombs has been a major result in discrete geometry. If these mosaics were spread out on the $[\bar n, \bar v]$ symbolic plane,
that would certainly imply that the associated values of the harmonic degree $\bar h$ cover a very broad range. However, this is not the case: all values of $\bar h$ are in the range
$3.31 \leq \bar h \leq 4$.  In addition, we also computed the values of $\bar h$ associated with the 28 dual mosaics, hyperplane random mosaics, the Poisson-Voronoi and Poisson-Delaunay random mosaics
and found that for the total of all the 60 mosaics the range is the same (cf. Table \ref{tab:1} in the Appendix). The indicated narrow range for the harmonic degree implies that  on the $[\bar n, \bar v]$ symbolic plane the points corresponding to these mosaics appear to accumulate on a narrow strip (cf. Figure \ref{fig:1}).

While we can not offer a full explanation of this phenomenon, we think that the concept of the harmonic degree may help to explain its essence. In particular, we introduce
\begin{conj}\label{con:1}
For any normal, face-to-face mosaic $\mathcal {M}$ in $\Re ^d$, we have $\bar h (\mathcal{M}) \in (d, 2^{d-1}]$.
\end{conj}

Or in a more detailed form:

\begin{prob}
Let $S_d$ denote the set of points of the $[\bar{n}, \bar{v} ]$ symbolic plane corresponding to all normal, face-to-face mosaics in $\Re^d$.
Prove or disprove that $S_d$ is bounded for all integers $d \geq 2$. More specifically, what are the sets $S_d$ for small values of $d$?
\end{prob}

\noindent To build intuitive support for Conjecture \ref{con:1} we will show in Section \ref{s:p1} that the interval indicated in the Conjecture has indeed some significance: we demonstrate mosaics 
corresponding to the lower and upper endpoints (the former understood as a limit outside the interval) and we also prove
\begin{thm}\label{thm:interval}
For all $\bar h ^{\star} \in (d, 2^{d-1}]$, there is a normal, face-to-face mosaic $\mathcal {M}$ in $\Re ^d$ satisfying $\bar h (\mathcal{M})=\bar h^{\star}$.
\end{thm}
\noindent Also, as a small step towards establishing the Conjecture, we prove
\begin{prop}\label{prop:h_estimate}
For any normal, face-to-face mosaic $\mathcal{M}$ in $\Re^d$, we have $ \bar{h}(\mathcal{M}) \geq \frac{d+1}{2}$.
Furthermore, if $d=3$, then  $ \bar{h}(\mathcal{M}) \geq \frac{28}{13} $.
\end{prop}
We provide the general formulae governing 3D mosaics in Section \ref{s:Schneidermatrix}. Next, we prove Theorems \ref{thm:harmonic} and \ref{thm:interval} in Section \ref{s:p1}.  Section \ref{qvary} discusses  non-Euclidean mosaics  and non-face-to-face mosaics in $d=2$ and $d=3$ dimensions. In Section \ref{s:sum} we draw conclusions.

\begin{figure}[ht]
\begin{center}
\includegraphics[width=\textwidth]{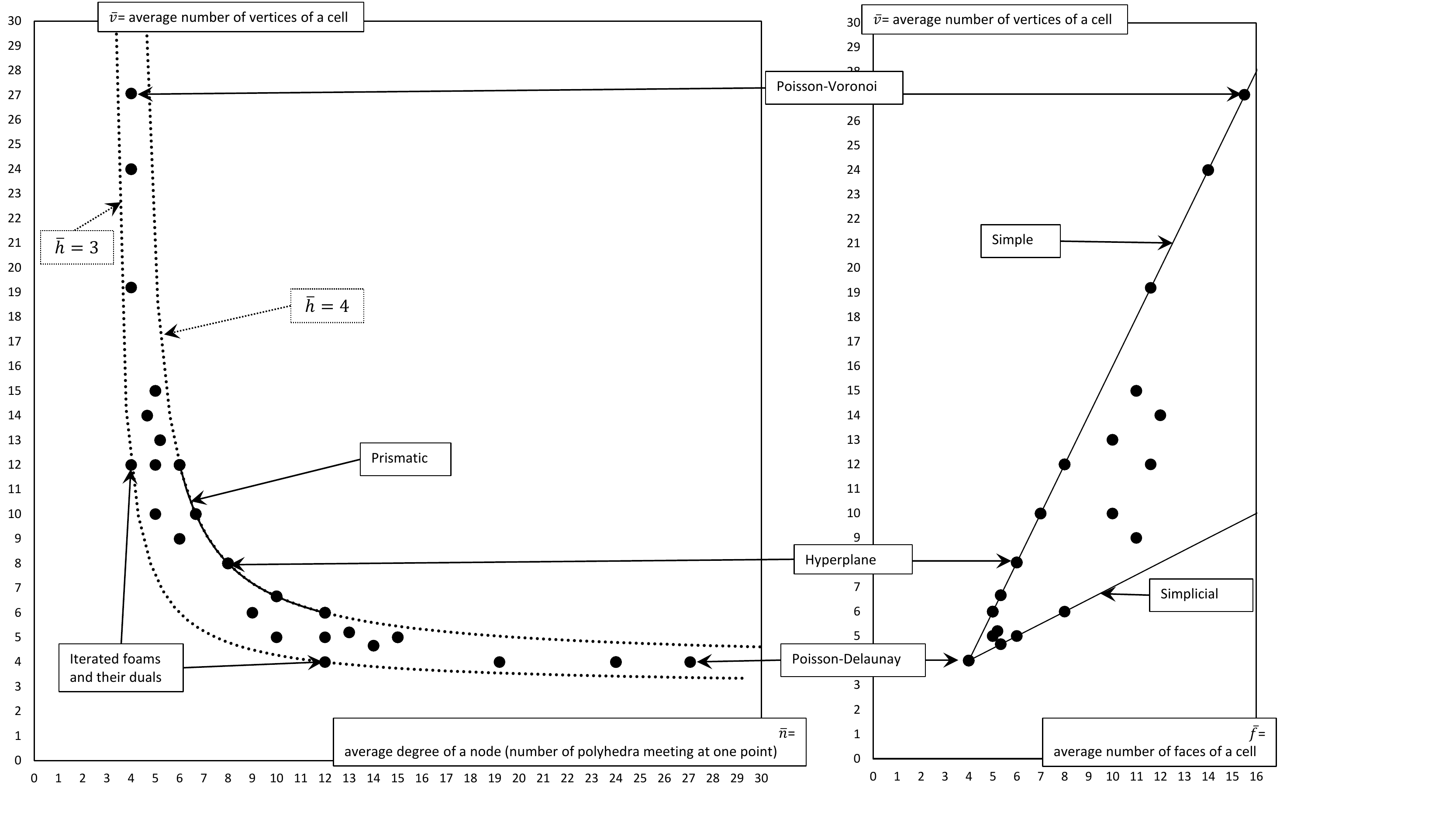}
\caption{The 28 uniform honeycombs, their duals, the hyperplane mosaics, the Poisson-Voronoi and Poisson-Delaunay mosaics, iterated foams and their duals (for details on the latter see Section \ref{ss:foam}) shown as black dots on the symbolic plane $[\bar n, \bar v]$ (left) and on the plane $[\bar f,\bar v]$, where $\bar{f}$ denotes the average number of faces of a cell of the mosaic (right).   For detailed numerical data see Table \ref{tab:1} in the Appendix. Continuous curve on the left panel shows prismatic mosaics.  Dotted lines represent the $\bar h =3$ and $\bar h=4$ curves, illustrating Conjecture \ref{con:1}. Continuous straight lines on the right panel correspond to simple and simplicial polyhedra.}\label{fig:1}
\end{center}
\end{figure}

\section{General formulae defining 3D mosaics}\label{s:Schneidermatrix}

In \cite{Schneider}, for any $0 \leq i,j \leq d$ and for any random mosaic $\mathcal{M}$ in the $d$-dimensional Euclidean space $\Re^d$, the quantity $n_{ij}$ is defined as the number of $j$-faces of a typical $i$-face of $\mathcal{M}$ if $j \leq i$, and as the number of $j$-faces containing a typical $i$-face of $\mathcal{M}$ if $j > i$.
Relations between these quantities are described in \cite[Theorem 10.1.6]{Schneider} for the case $d=2$, and in 
\cite[Theorem 10.1.7]{Schneider} for the case $d=3$.
Here we use elementary, combinatorial arguments to show that these relations hold for any convex mosaic in $\Re^3$.

 Throughout this section, let $\mathcal{M}$ be a convex, face-to-face, normal mosaic in $\Re^d$.
Then we may define $n_{ij}(\mathcal{M})$ as the \emph{average} number of $j$-faces contained in or containing a given $i$-face, if $j \leq i$ or $j > i$, respectively. If it is clear which convex mosaic $\mathcal{M}$ we refer to, for brevity we may use the notation $n_{ij}=n_{ij}(\mathcal{M})$.
Here we assume that the average of any $n_{ij}$, for all values of $i$ and $j$ exists.

\begin{thm}\label{thm:nij}
For any convex mosaic $\mathcal{M}$ in $\Re^3$ satisfying the conditions in the previous paragraph, we have
\begin{equation}\label{eq:nij}
[n_{ij}(\mathcal{M})]=
\left[
\begin{array}{cccc}
1 & \frac{(f-2) n}{v} + 2 & \frac{(f-2) n}{v} + n & n \\
2 & 1 & \frac{2(v+f-2)n}{(f-2)n + 2v} & \frac{2(v+f-2)n}{(f-2)n + 2v}\\
\frac{2(v-2)}{f} + 2 & \frac{2(v-2)}{f} + 2 & 1 & 2\\
v & v+f-2 & f & 1
\end{array}
\right],
\end{equation}
where $v=n_{30}$, $f=n_{32}$ and $n=n_{03}$.
\end{thm}

\begin{proof}
Clearly, $n_{ii}=1$ for $i=1,2,3,4$, and since each face belongs to exactly two cells, and each edge has exactly two endpoints, we have $n_{23}=n_{10}=2$.
The formula $n_{31}=v+f-2$ follows from applying Euler's formula for each cell of $\mathcal{M}$, and observing that then the same formula holds for the average numbers of faces, edges and vertices of a cell.

Let $r$ be sufficiently large, and let $B_r$ be the Euclidean ball, centered at the origin $o$ and with radius $o$. Let $N_v(r)$, $N_e(r)$, $N_f(r)$ and $N_c(r)$ denote the number of vertices, edges, faces and cells of $\mathcal{M}$ in $B_r$, respectively.

Note that if $r$ is sufficiently large, the sum of the numbers of edges of all faces in $B_r$ is approximately $N_f(r) n_{21}$, and since almost all face in $B_r$ belongs to exactly two cells in $B_r$, and each edge of a given cell belongs to exactly two faces of the cell, we have that this quantity is approximately equal to $N_c(r) n_{31}$. On the other hand, the sum of the numbers of faces the cells in $B_r$ have is approximately $N_c(r) f \approx 2 N_f(r)$.
More specifically, we have
\[
\frac{f}{2} = \lim_{r \to \infty} \frac{N_f(r)}{N_c(r)} = \frac{n_{31}}{n_{21}},
\]
which readily yields that $n_{21} = n_{20} = \frac{2(v+f-2)}{f}$.

Note that the number of cell-vertex incidences in $B_r$ is approximately $N_v(r) n \approx N_c(r) v$. Furthermore, for any incident cell-vertex pair $C,v$, the number of faces that contain the vertex and is contained in the cell is equal to the number of edges with the same property. Let us denote this common number by $\deg_C(v)$, which then denotes the degree of the vertex $v$ in the cell $C$. We compute the approximate value of the quantity $Q = \sum_{C \subset B_r, v \in C} \deg_C(v)$ in two different ways.

On one hand, we have
\[
Q = \sum_{C \subset B_r} 2 e(C) \approx 2 n_{31} N_c(r),
\]
where $e(C)$ denotes the number of the edges of the cell $C$.
On the other hand, since any face belongs to exactly two cells, we also have
\[
Q = \sum_{v \in B_r} \sum_{C \subset B_r, C \ni v} \deg_C(v)
\approx 2 n_{02} N_v(r).
\]
More precisely, we have obtained that
\[
\frac{n}{v} = \lim_{r \to \infty} \frac{N_c(R)}{N_v(r)} = \frac{n_{02}}{n_{31}},
\]
which implies the expression for $n_{02}$.

The value of $n_{01}$ can be obtained from the application of Euler's formula for the vertex figure at every node.
Finally, the value of $n_{12}=n_{13}$ can be obtained from the values of the other $n_{ij}$s like the value of $n_{20}=n_{21}$.
\end{proof}

\begin{rem}
By Theorem~\ref{thm:nij}, it seems that many combinatorial properties of the convex mosaic $\mathcal{M}$ are determined by three parameters, say by $v,f,n$. One may try to find upper and lower bounds for these values by observing that each entry in $[n_{ij}]$ has a minimal value: e.g. $n_{30}, n_{32}, n_{03}, n_{01} \geq 4$ and $n_{12},n_{21} \geq 3$. Nevertheless, solving these inequalities puts no restriction on the values of $v,f,n$, apart from the trivial inequalities $n,v,f \geq 4$.
It is worth noting that in contrast, for convex polyhedra (or, in other words, for convex spherical $2$-dimensional mosaics, cf. Remark~\ref{rem:dual_mosaic}), the sharp inequalities $\frac{v}{2}+2 \leq f \leq 2v+4$ \cite{Steinitz1, Steinitz2} are immediate consequences of the fact that each face of the polyhedron has at least $3$ vertices, and each vertex belongs to at least $3$ faces.
\end{rem}

\begin{rem}
Note that if $\mathcal{M}$ is a convex mosaic in $\Re^3$ and a convex mosaic $\mathcal{M}^{\circ}$ is its dual, then
\[
n_{ij}(\mathcal{M}^{\circ}) = n_{(3-i)(3-j)} (\mathcal{M}),
\]
for all $0 \leq i,j \leq 3$.
\end{rem}

\section{Proof of  the Theorems } \label{s:p1}
\subsection{Proof of Theorem \ref{thm:harmonic} and the volumes of polar domains}\label{ss:eupolar}

\begin{proof}

First we show that in case of Euclidean mosaics (in arbitrary dimensions) the harmonic degree may be interpreted as the averaged inverse sum of two angles linked by polarity, one of which is the internal vertex angle of a cell.

Consider a convex  face-to-face mosaic $\mathcal{M}$ in $\Re^d$. For any cell $C$ in $\mathcal{M}$ and any vertex $p$ of $C$, let $I(C,p) \subset \mathbb{S}^{d-1}$ denote the set of unit vectors such that the rays in the direction of a vector in $I(C,p)$ and starting at $p$ contain points of $C \setminus \{ p \}$. Furthermore, let $E(C,p) \subset \mathbb{S}^{d-1}$ denote the set of outer unit normal vectors of $C$ at $p$. Then, by definition, we have 
that $E(C,p)$ is the polar $(I(C,p))^{\circ}$ of $I(C,p)$. Let us denote the spherical volumes of $E(C,p)$ and $I(C,p)$ by $\Omega_E(C,p)$ and $\Omega_I(C,p)$, respectively, and let $\bar \Omega_E, \bar \Omega_I$ define the average values of $\Omega_E(C,p)$ and $\Omega_I(C,p),$ respectively, over all incident pairs $C$ and $p$.

Note that for any cell $C$, the family of sets $E(C,p)$, where $p$ runs over the vertices of $C$, is clearly a spherical mosaic of $\mathbb{S}^{d-1}$, and thus, the total area of the members of this family is the surface area $S_{d-1}$ of the sphere. The same statement holds for the family of sets $I(C,p)$, where $C$ runs over the cells containing a given node $p$.

Now, consider a large ball $B$ of space with radius $r$, and denote by $N_c$ and $N_v$ the numbers of cells and nodes of $\mathcal{M}$ in $B$, respectively. Then, for the number $k(r)$ of incident pairs of cells and vertices in $B$  we have
\begin{equation} \label{polar0}
k(r) \approx N_c \bar{v} \approx N_v \bar{n}.
\end{equation}  
The sums  $\omega _I, \omega _E$ of $\Omega _I(C,p)$ and $\Omega _E(C,p)$ (over all pairs of cells $C$ and incident vertices $p$ in $B$) may be written as:
\begin{equation}\label{polar1}
\begin{array}{rcl}
\omega _I & \approx & N_v S_{d-1} \\
\omega _E & \approx & N_c S_{d-1},
\end{array}
\end{equation}
so,  for the averages $\bar \Omega _I=\lim_{r \to \infty} (\omega _I/k), \bar \Omega_E = \lim_{r\to \infty} (\omega _E/k)$ we get
\begin{equation}\label{polar2}
\bar \Omega _I \bar n= \bar \Omega_E \bar v =S_{d-1}.
\end{equation}
Thus, by (\ref{polar2}) we have
\begin{equation}
\bar{\Omega} = \bar \Omega _I + \bar \Omega _E = \frac{S_{d-1}}{\bar{v}} + \frac{S_{d-1}}{\bar{n}},
\end{equation}
implying that
\begin{equation}\label{eq:harmonic_geometric}
\bar{h}\bar{\Omega} = S_{d-1}.
\end{equation}
\end{proof}
\begin{rem}\label{rem:geometric_Euclidean}
Substituting the value of $S_{d-1}$ into (\ref{eq:harmonic_geometric}), we obtain that for planar mosaics $\bar{h}=\frac{2\pi}{\bar{\Omega}}$, and for mosaics in $\Re^3$ $\bar{h} = \frac{4 \pi}{\bar{\Omega}}$.
\end{rem}

\begin{rem}
If $d=2$, then at each vertex we have $\Omega_I(C,p)+\Omega_E(C,p) = \pi$, implying $\bar \Omega =\pi$ and this, via  equation (\ref{eq:harmonic_geometric}) yields $\bar h=2$. 
\end{rem}

\begin{rem}
As we observed, for any node $p$ and any cell $C$ incident to $p$, we have
\[
\sum_{ \{ p : p \in C\} } \Omega_E(C,p) = \sum_{\{ C : p \in C \} } \Omega_I(C,p) = S_{d-1}.
\]
While the equality $$\sum_{\{ C : p \in C \} } \Omega_E(C,p) = S_{d-1}$$ does not hold in general, it does hold in case of hyperplane mosaics.
\end{rem}

\begin{rem}
Clearly, the inequalities $\bar{v}, \bar{n} \geq 4$ readily imply $\bar{h} \geq 2$, and by Theorem~\ref{thm:harmonic}, $\bar{\Omega} \leq \frac{1}{2} S_{d-1}$.
This inequality is also an immediate consequence of the well-known result of Gao, Hug and Schneider \cite{GHSch}, stating that for any spherically convex set $A$ of a given spherical volume, the volume of its polar $A^{\circ}$ is maximal if $A$ is a spherical cap of the given spherical volume.
\end{rem}

\subsection{Proof of Theorem  \ref{thm:interval} and the range of the harmonic degree}\label{ss:foam}

\subsubsection{Mosaics with high harmonic degree: hyperplane mosaics}
If $\mathcal{M}$  is generated by dissecting $\Re ^d$ with $(d-1)$-dimensional hyperplanes then it is called a \emph{hyperplane mosaic} \cite{Schneider}.  An elementary computation shows that the harmonic degree of a normal mosaic, generated by hyperplanes in general position, is
\begin{equation}\label{hyperplane}
\bar h =2^{(d-1)}.
\end{equation}
These mosaics appear to have the highest harmonic degree. They are certainly not the only mosaics with $\bar h = 2^{d-1}$, though. In $d=2$ dimensions all convex mosaics have $\bar h =2$ and in $d=3$ dimensions we have the continuum of prismatic mosaics with $\bar h =4$.

\subsubsection{Mosaics with low harmonic degree: iterated foams and their duals}\label{subsec:foam}
Here we define mosaics which appear to have extremely low harmonic degrees.

Consider a Euclidean mosaic $\mathcal{M} = \mathcal{M}_0$ with $\bar{n}_0=d+1$ as the average degree of nodes; we remark that such mosaics exist in all dimensions, we construct the dual of such a mosaic in the proof of Theorem~\ref{thm:interval}.
In addition, we assume that the edge lengths of the mosaic are uniformly bounded; that is there are some $0 < a < b$ such that
the value of each such quantity is between $a$ and $b$, and we assume the same about the angles between any two faces of $\mathcal{M}$.

Note that since all $d$-dimensional convex polytopes with $d+1$ vertices are simplices,
the vertex figures of `almost all' nodes of $\mathcal{M}$ are simplices.
Now, for each node having a simplex as a vertex figure, replace the node with its vertex figure.
More precisely, if $p$ is a node whose vertex figure is a simplex, define the cell $C_p$ as the convex hull of the points of the edges starting at $p$, at the distance $\varepsilon > 0$ from $p$ for some fixed value of $\varepsilon$ independent of $p$, and replace each cell $C$ containing $p$ with the closure of
$C \setminus C_p$. Then,if this process is carried out simultaneously at all nodes $p$, we obtain a convex, face-to-face mosaic $\mathcal{M}_1$, which, under our condition, is normal.
Applying this procedure $k$ times we obtain the convex, face-to-face, normal mosaic $\mathcal{M}_k$. We will call the $k \to \infty$ limit of such a sequence a $d$-dimensional \emph{iterated foam},
referring to the fact that in a physical foam in $d=2$ and $d=3$ dimensions we always have $\bar n = d+1$.

Clearly, for all $k \geq 1$, we have $\bar{n}_k = \bar{n}(\mathcal{M}_k) = d+1$.
Consider a sufficiently large region of space. Then the number of vertex-cell incident pairs in $\mathcal{M}$ is approximately $N_c \bar{v} \approx N_v \bar{n} = N_v (d+1)$,
where $N_c$ and $N_v$ denote the numbers of the cells and the nodes of the mosaic in this region, respectively.
An elementary computation yields that for the mosaic $\mathcal{M}_1$, this number is approximately $N_c d \bar{v} + (d+1) N_v \approx (d+1) N_c \bar{v}$, and
the number of cells of $\mathcal{M}_1$ in this region is about $N_c + N_v \approx \left( 1 + \frac{v}{d+1} \right) N_c$.
Thus, taking limit, we obtain that the average degree of a cell of $\mathcal{M}_1$ is $\bar{v}(\mathcal{M}_1) = \frac{(d+1)^2 \bar{v}}{d+1+\bar{v}}$.
Setting $\bar{v}_k = \bar{v}(\mathcal{M}_k)$, we similarly obtain the recursive formula $\bar{v}_{k+1} = \frac{(d+1)^2 \bar{v}_k}{d+1+\bar{v}_k}$ for all nonnegative integers $k$.

An elementary computation yields that $\left| \bar{v}_{k+1} - d(d+1)\right| = \frac{(d+1) \left| \bar{v}_k-d(d+1)\right|}{d+1+\bar{v}_k} \leq \frac{1}{2} \left| \bar{v}_k-d(d+1) \right|$ for all $\bar{v}_k \geq d+1$. This implies that for any initial value $\bar{v} \geq d+1$, the sequence $\{ \bar{v}_k \}$ converges to $d(d+1)$, and thus, the sequence $\{ h(\mathcal{M}_k)\}$ converges to $\frac{d(d+1)^2}{d(d+1)+(d+1)}=d$.

We note that the above procedure can be dualized. In this case, starting with a mosaic in which all cells are simplices, in each step we divide the cell into regions by taking the convex hulls of a given interior point of the cell with each facet of the cell.  Lines 31 and $31'$ of Table \ref{tab:1} in the Appendix summarize the main parameters of these iterated mosaics.

\begin{figure}[ht]
\begin{center}
\includegraphics[width=\textwidth]{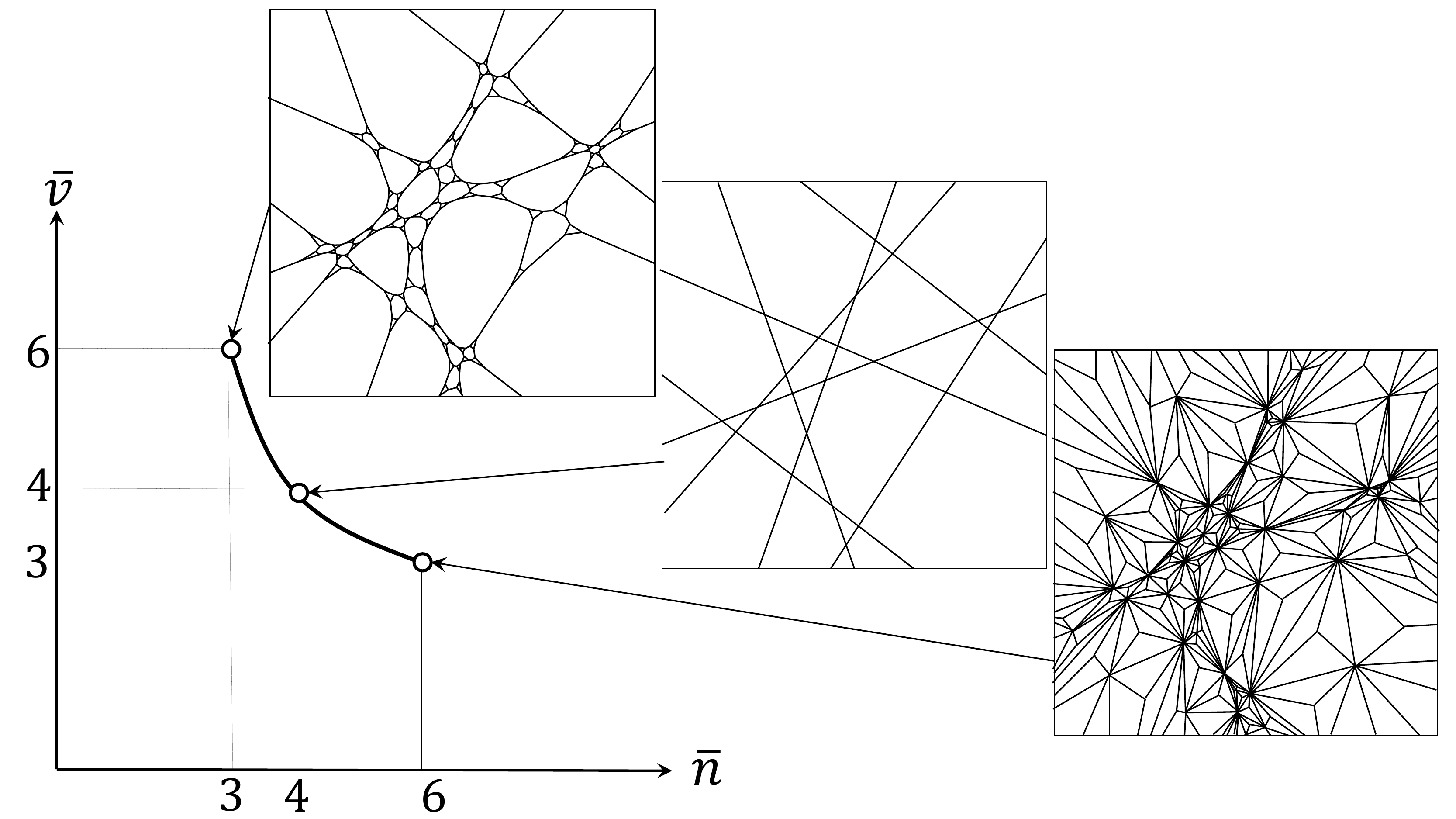}
\caption{Illustration of an iterated foam and its dual in $d=2$ dimensions.  We used the hyperplane mosaic (middle panel) as initial condition and ran $k=2$ iterative steps both in the direction of iterated foams (upper panel) as well as their duals (lower panel). Note that in $d=2$ dimensions these iterative steps change both $\bar n$ and $\bar v$, however, the harmonic degree
remains constant at $\bar h =2$. Also note that in higher dimensions hyperplane mosaics may not be used as initial conditions for these iterations. Iterated foams and their duals in $d=3$ dimensions are shown in the symbolic plane on Figure \ref{fig:1}.}\label{fig:foam}
\end{center}
\end{figure}

\begin{rem}
We note that for planar mosaics the iterating process (and also its dual), can be extended to any mosaic in a natural way, and after one iteration step the degree of every node (in case of its dual the degree of every cell) is equal to $3$. This is illustrated in Figure \ref{fig:foam} where we iterate a (finite domain of a) hyperplane mosaic for $k=2$ steps in both directions. 
\end{rem}

\subsubsection{Proof of Theorem \ref{thm:interval}}

Let $\mathcal{M}$ be the standard cubic mosaic in $\Re^d$, whose vertices are the points of the integer lattice $\mathbb{Z}^d$, and note that $\bar{h}=2^{d-1}$.
We define a new lattice $\mathcal{M}'$ as the first barycentric subdivision of $\mathcal{M}$. In this lattice the centroid of each face of $\mathcal{M}$ is a vertex of $\mathcal{M}'$, and cells correspond to \emph{flag}s of $\mathcal{M}$, where a flag is a sequence $F_0 \subset F_1 \subset F_2 \subset \ldots \subset F_d$ of faces of $\mathcal{M}$, with $\dim F_i = i$ for all values of $i$. In this case the cell associated to the flag is the convex hull of the centroids of the $F_i$s.

We compute the harmonic degree of $\mathcal{M}'$. Note that since every cell of $\mathcal{M}'$ is a simplex, we have $\bar{v}(\mathcal{M}') = d+1$.
First, observe that since any cell of $\mathcal{M}$ has $2d$ facets and by the fact the every face of a cube is a cube, choosing the faces of a flag in the order $F_d, F_{d-1}, \ldots, F_0$, we have that the number of flags belonging to any given cube in the mosaic is $2d \cdot (2d-2) \cdot \ldots \cdot 2 = 2^d d!$.
We note that the same quantity can be obtained if we choose the faces of a flag in the opposite order. In this approach first we choose a vertex of the cell, then we extend this point to an edge parallel to a chosen coordinate axis, which then can be extended to a $2$-face choosing another coordinate axis. In this way the number of flags is equal to the product of the number of vertices ($2^d$), and the number of permutations of the $d$ coordinate axes ($d!$).

Applying arguments similar to these two counting arguments, one may show that each $i$-face belongs to $2^i i! (d-i)!$ flags within one cell, and as each $i$-face belongs to $2^{d-i}$ cells, the total number of flags an $i$-face belongs to is $2^d i! (d-i)!$. Furthermore, the proportion of the $i$-faces compared to the number of cells is $\binom{d}{i}$.
Thus, the average degree of a node in the barycentric subdivision of the cubic lattice is
\[
\bar{n}(\mathcal{M}')= \frac{\sum_{i=0}^d 2^d i! (d-i)! \binom{d}{i}}{\sum_{i=0}^d \binom{d}{i}} = (d+1)!,
\]
implying
\begin{equation}\label{eq:h'}
\bar{h}' = \bar{h}(\mathcal{M}') = \frac{(d+1)!}{1+d!},
\end{equation}
where an elementary computation yields that $d \leq \frac{(d+1)!}{1+d!} < d+1$ for all $d \geq 1$.

\emph{Case 1}, $\bar{h}' = \frac{(d+1)!}{1+d!} \leq \bar{h}^{\star} \leq \bar{h} = 2^{d-1}$.
We construct a mosaic with harmonic degree $\bar{h}^{\star}$.
To do it, we use four types of layers.

A first type layer is a translate of the part of the cubic lattice between the hyperplanes $\{ x_d = 0\}$ and $\{ x_d = 1 \}$.
A second type layer is the same part of the \emph{subdivided} cubic lattice $\mathcal{M}'$.
For the third type, we take the translates of a partial subdivision of the cubic lattice: each cube in the strip between $\{ x_d = 0\}$ and $\{ x_d = 1 \}$ is subdivided by the centroids of all faces apart from those in the hyperplane $\{ x_d = 1\}$. The fourth type layers are the reflected copies of third type layers about the hyperplane $\{ x_d = 0\}$.

The building bricks of the mosaic are strips $S(k,l)$ of width $k+l+2$, where $k$ and $l$ are positive integers. Here $S(k,l)$  consisting of $k$ first type, $1$ fourth type, $l$ second type and $1$ third type layer in this consecutive order, where the layers are attached in a face-to-face way. Observe that if $k$ and $l$ are sufficiently large, then the harmonic degree of a strip $S(k,l)$ is approximately $\frac{k}{k+l} \bar{h} + \frac{l}{k+l} \bar{h}'$.

Since $\bar{h}' \leq \bar{h}^{\star} \leq \bar{h}$, there is some $0 \leq \lambda \leq 1$ such that $\bar{h}^{\star} = \lambda \bar{h} + (1-\lambda) \bar{h}'$.
Let $\{(k_m,l_m) \}$ be a sequence of pairs of positive integers such that $k_m+l_m \to \infty$, and $\frac{k_m}{k_m+l_m} \to \lambda$.

We define the mosaic $\mathcal{M}^{\star}$ as follows. Consider a strip $S_1=S(k_1,l_1)$. Attach two copies of $S(k_2,l_2)$ to the two bounding hyperplanes of
$S_1$ in a face-to-face way, to obtain $S_2$ as the union of these three strips. Then $S_3$ is constructed by attaching two copies of $S(k_3,l_3)$ to $S_2$ in a face-to-face way. Continuing this procedure, we obtain the mosaic $\mathcal{M}^{\star}$ as the limit of the strip $S_m$, where $m \to \infty$.
Then the harmonic index of $\mathcal{M}^{\star}$ is $\bar{h}^{\star}$.

\emph{Case 2}, $d < \bar{h}^{\star} < \bar{h}'$.
Observe that in the subdivided cubic mosaic $\mathcal{M}'$ defined above, every cell is a simplex. Thus, we may apply the dual of the algorithm discussed in 
Subsection~\ref{subsec:foam}, namely in each step we divide each cell $C$ into $d+1$ new cells by taking the convex hulls of a given interior point of $C$ and the facets of $C$. Let us denote by $\mathcal{M}_k$ and $\bar{h}_k$ the mosaic obtained by $k$ subsequent applications of this procedure, and its harmonic degree, respectively.
Then the sequence $\{ \bar{h}_k\}$ tends to $d$, and thus, there is a smallest value of $k$ such that $\bar{h}^{\star}$ is in the interval $[ \bar{h}_{k+1}, \bar{h}_k]$. To construct a suitable mosaic $\mathcal{M^{\star}}$ with harmonic degree $\bar{h}^{\star}$, we follow the idea of the proof in Case 1, and divide  only a part of the cells of $\mathcal{M}_k$ into new cells.

\begin{rem}
We remark that if the cells are not assumed to be convex, we may construct face-to-face mosaics with harmonic degree arbitrarily close to $2$. Indeed, starting with a cubic mosaic we can break each facet by choosing a point from inside and slightly moving it. Thus, for any $h \in (2, 2^{d-1}]$, there is a face-to face mosaic with harmonic degree equal to $h$.
\end{rem}

\subsubsection{Proof of Proposition~\ref{prop:h_estimate}}

The first part of the proposition follows from the trivial estimates $\bar{v}, \bar{n} \geq d+1$.
To prove the second part, we need a lemma. We note that the minimum number of tetrahedra such that each convex polyhedron with $k$ vertices can be decomposed into is not known. This fact and the idea of the proof of Lemma~\ref{lem:tetralas} can be found in \cite{approx}.

\begin{lem}\label{lem:tetralas}
Any convex polyhedron $P$ in $\Re^3$ with $v$ vertices can be decomposed into at most $2v-7$ tetrahedra.
\end{lem}

\begin{proof}
Let the faces of $P$ be $G_1, \ldots, G_f$, and let $f_i$ denote the number of edges of $G_i$.
Then $\sum_{i=1}^f f_i = 2e = 2v+2f-4$, where $e$ is the number of edges of $P$. 

Let $p$ be any vertex of $P$. Let us triangulate each face of $P$ containing $p$ by the diagonals starting at $p$, and all other faces of $P$ by the diagonals starting at an arbitrary vertex of the face. Then the number of all triangles is $\sum_{i=1}^f (f_i-2)=2v+2f-4-2f=2v-4$.
Since each face contains at least one triangle, and each vertex belongs to at least three faces, the number $m$ of triangles in the faces not containing $p$ is at most
$m \leq 2v-7$. Now, if these triangles are $T_1,T_2, \ldots, T_m$, then the tetrahedra $\conv (\{p\} \cup T_i)$, $i=1,2,\ldots,m$ is a required decomposition of $P$.
\end{proof}

Consider a mosaic $\mathcal{M}$ in $\Re^3$, and a sufficiently large region. Let $N_c$ and $N_v$ denote the numbers of cells and nodes of $\mathcal{M}$ in this region. For any cell $C_i$, let $v_i$ denote the number of vertices of $C_i$. Then the number of cell-vertex incidences in this regions is approximately
$\sum_{i} v_i \approx N_c \bar{v} \approx N_v \bar{n}$.

By Lemma~\ref{lem:tetralas}, these cells can be decomposed into at most $\sum_{i} (2v_i-7) \approx 2 N_c \bar{v} - 7 N_c$ tetrahedra.
It is well known that the  sum of the internal angles of any tetrahedron is greater than $0$ and less than $2\pi$ \cite{Gaddum}. Thus, the sum of all the internal angles
of the cells is at most $4\pi N_c \bar{v} - 14\pi N_c$. On the other hand, this sum is approximately equal to the product of the number of nodes and the total angle of a sphere; that is $4\pi N_v$. Thus, apart from a negligible error term, we have $4 \pi \frac{N_c \bar{v}}{\bar{n}} \leq 4\pi N_c \bar{v} - 14\pi N_c$. Taking a limit, we obtain that $\frac{2 \bar{v}}{\bar{n}} \leq 2 \bar{v} - 7$, implying that $\bar{n} \geq \frac{2 \bar{v}}{2\bar{v}-7}$.
Since $\bar{n} \geq 4$ clearly holds, we have that
\begin{equation}
\bar{n} \geq \max \{ 4, \frac{2 \bar{v}}{2\bar{v}-7} \}.
\end{equation}
It is an elementary computation to check that $4 \geq \frac{2 \bar{v}}{2\bar{v}-7}$ if and only if $\bar{v} \geq \frac{14}{3}$.
Since for any fixed value of $\bar{n}$, $\bar{h}$ is minimal at the minimal value of $\bar{v}$ it follows that under the condition that $\bar{v} \geq \frac{14}{3}$, we have $\bar{h} \geq \frac{28}{13}$. Furthermore, if $4 \leq \bar{v} \leq \frac{14}{3}$, then
$\bar{h}= \frac{2\bar{v}}{2\bar{v}-5}$, which is minimal if $\bar{v}=\frac{14}{3}$, and thus, $\bar{h} \geq \frac{28}{13}$ also in this case.

\section{Non-Euclidean and non face-to-face mosaics}\label{qvary}

\subsection{Non-Euclidean mosaics}
Mosaics, convex mosaics, and all notions described in Subsection~\ref{subsec:history}, excluding the notions of average degrees of cells and vertices, can be defined in a natural way for spherical and hyperbolic spaces as well.
For spherical space, this includes average degrees as well; because of the compactness of the space it is even possible to avoid the usual limit argument applied to compute these values in $\Re^d$.

On the other hand, defining average values in hyperbolic space seems problematic. Indeed, it is well known that under rather loose restrictions, in a packing of congruent balls in $\mathbb{H}^d$, the number of balls intersecting the boundary of a hyperbolic ball $B$ of large radius is \emph{not} negligible compared to the number of balls contained in $B$. This phenomenon is explored in more details, for instance, in \cite{FTK}, and can be generalized for the numbers of cells of a normal mosaic in a natural way.

A straightforward solution to this problem is to examine only regular mosaics, in which the degree of every cell, and the degree of every vertex is equal, which offers a natural definition for $\bar{n}$ and $\bar{v}$. We do this in Subsection~\ref{ss:noneuclidean_regular}.
To circumvent this problem in a more general way, we use the geometric interpretation of harmonic degree for mosaics in $\Re^d$, appearing in Subsection~\ref{ss:eupolar}; this interpretation, in particular, provided a different proof of the fact that harmonic degree is $2$ for every planar Euclidean mosaic.

In Subsection~\ref{ss:noneupolar} we generalize this geometric interpretation for mosaics in $\S^2$ and $\HH^2$, and show that for spherical mosaics it coincides with the original definition of harmonic degree. Finally, we show that this value is less than $2$ for any spherical mosaic, and it is at least $2$ for any hyperbolic mosaic, using any reasonable interpretation of average.

\subsubsection{Non-Euclidean regular honeycombs in $\mathbb{S}^d, \mathbb{H}^d$ for $d=2,3$. }\label{ss:noneuclidean_regular}

Here we show that in $d=2$ dimensions, Euclidean mosaics separate regular spherical 
mosaics from regular hyperbolic mosaics on the $[\bar n, \bar v]$ symbolic plane.

While Plato's original idea of filling the Euclidean space with regular solids proved to be incorrect, if we relax the condition that the embedding space has no curvature then all Platonic solids may fill space by what we call a \emph{regular honeycomb}.  We briefly review these mosaics to show how they are represented in our notation and how their harmonic degrees are spread.

Let $\mathcal{M}$ be a honeycomb in a space of constant curvature of dimension $d$. A sequence $F_0 \subset F_1 \subset \ldots \subset F_d$, where $F_i$ is an $i$-dimensional face of $\mathcal{M}$, is called a \emph{flag} of $\mathcal{M}$ (cf. the proof of Theorem \ref{thm:interval}). We say that $\mathcal{M}$ is \emph{regular}, if for any two flags of $\mathcal{M}$ there is an element of the symmetry group of $\mathcal{M}$ that maps one of them into the other one.
In particular, if $\mathcal{M}$ is a regular planar mosaic, then the cells of $\mathcal{M}$ are congruent regular $p$-gons, and at each node, an equal $q$ number of edges meet at equal angles. In this case $\{ p, q \}$ is called the \emph{Schl\"afli symbol} of $\mathcal{M}$. It is well known that up to congruence, for any values $p,q \geq 3$, there is a unique regular mosaic with Schl\"afli symbol $\{p,q\}$ (cf. \cite{Schulte1}). This mosaic if spherical if $p=3$ and $q=3,4,5$ or if $q=3$ and $p=3,4,5$, Euclidean
if $\{p,q \} = \{3,6 \}, \{4,4\}, \{6,3\}$, and hyperbolic otherwise. We note that the five regular spherical honeycombs correspond to the five Platonic polyhedra.
The Schl\"afli symbol of a higher dimensional mosaic can be defined recursively: it is $\{ p_1, p_2, \ldots, p_d \}$ if the Schl\"afli symbol of its cells are 
$\{ p_1, p_2, \ldots, p_{d-1} \}$ (which must correspond to a regular spherical mosaic), and the intersection of $\mathcal{M}$ with any sufficiently small sphere centered at a node of $\mathcal{M}$ is the regular spherical mosaic $\{ p_2, p_3, \ldots, p_d \}$.

In $d=2$ dimensions the $\bar h=2$ curve defines a partition of the $\Z^2$ grid on the $[\bar n, \bar v]$ symbolic plane  with the constraints $\bar n, \bar v \geq 3$. For a regular mosaic with Schl\"afli symbol $\{ p,q\}$, set $\bar{v}=p$, $\bar{n}=q$, or equivalently, $\bar{h}= \frac{pq}{p+q}$.
Then an elementary computation  (determining the sign of the quantity $\frac{1}{p}+\frac{1}{q} - \frac{1}{2}$ for all integers $p,q \geq 3$) shows that grid points on the $\bar h =2$ line correspond to regular Euclidean mosaics, grid points with $\bar h<2$ correspond to regular spherical mosaics and grid points with $\bar h >2$ correspond to regular hyperbolic mosaics.

In $d=3$ dimensions the $\bar h=4$ curve defines a partition of the $\Z^2$ grid on the $[\bar n, \bar v]$ symbolic plane in a similar sense, although
here only a finite number of grid points correspond to regular mosaics. We summarize these in Table \ref{tab:2}.

\begin{table} [ht] 
{
\begin{tabular}{| c |r||c|c|c|r|r|r|r|}
\hline
ID.  & Cell & Node &  Space& $\bar n$ & $\bar v$ &  $\bar h $ & Schl\"afli \\
\hline
\hline
 1 & cube & octahedron & Euclidean & 8 & 8 & 0.250 & $\{4,3,4\}$ \\
\hline
2 &  	 icosahedron & dodecahedron  & Hyperbolic	& 12	& 12	& 	0.167 & $\{3,5,3\}$ \\
\hline
3 &  	 dodecahedron & icosahedron  & Hyperbolic	& 20	& 20	& 	0.100 & $\{5,3,5\}$ \\
\hline
4 &  	 cube & icosahedron  & Hyperbolic	& 8	& 20	& 	0.175 & $\{4,3,5\}$ \\
\hline
5 &  	 dodecahedron & octahedron  & Hyperbolic	& 20	& 8	& 	0.175 & $\{5,3,4\}$ \\
\hline
6 &  	 tetrahedron & tetrahedron  & Elliptic	& 4	& 4	& 	0.500 & $\{3,3,3\}$ \\
\hline
7 &  	 octahedron & cube  & Elliptic	& 6	& 6	& 	0.333 & $\{4,3,4\}$ \\
\hline
8 &  	 cube &  octahedron  & Elliptic	& 8	& 4	& 	0.375 & $\{4,3,3\}$ \\
\hline
9 &  	 tetrahedron &  octahedron  & Elliptic	& 4& 8	& 	0.375 & $\{3,3,4\}$ \\
\hline
10 &  	dodecahedron & tetrahedron   & Elliptic	&  20 & 4	& 	0.300 & $\{5,3,3\}$ \\
\hline
11 &  	 tetrahedron & icosahedron  & Elliptic	&  4 & 20	& 	0.300 & $\{3,3,5\}$ \\
\hline

\end{tabular}
}
\vspace{0.5cm}
\caption{Regular honeycombs in $d=3$ dimensions}
\label{tab:2}
\end{table}

As we can observe, the  harmonic degree $\bar h$ of a mosaic appears to carry information both on the dimension and the curvature of the embedding space:
$\bar h=constant$ curves separate convex mosaics embedded in spaces with the same curvature sign  but different dimension and vice versa, they also separate regular mosaics embedded in spaces with the same dimension but different sign of curvature. Knowing one
of those parameters seems to permit us to obtain the other, based on the mosaic's harmonic degree.

\subsubsection{Non-Euclidean general face-to-face mosaics on $\mathbb{S}^2$ and $\mathbb{H}^2$}\label{ss:noneupolar}

Our goal is to extend the geometric interpretation of the harmonic degree to convex face-to-face mosaics on $\mathbb{S}^2$ and $\mathbb{H}^2$. First we describe how the duals of spherical mosaics can be constructed. To do this, first we compute the harmonic degree of spherical mosaics directly.

\begin{rem}\label{rem:dual_mosaic}
Clearly, projecting a convex polyhedron $P$ from an interior point to a sphere concentric to this point yields a spherical mosaic. Furthermore, in a spherical mosaic any two cells intersect in one edge, one vertex or they are disjoint. Using these properties it is easy to show that the edge graph of any spherical mosaic is $3$-connected and planar; such an argument can be found, e.g. in the proof of \cite[Claim 9.4]{BLNP}. By a famous theorem of Steinitz \cite{Steinitz2}, every $3$-connected planar graph is the edge graph of a convex polyhedron. Thus, up to combinatorial equivalence, we may regard a spherical mosaic as the central projection on $\mathbb{S}^2$ of a convex polyhedron $P$ containing the origin in its interior. This representation permits us to define the dual of a spherical mosaic associated to $P$ as the mosaic associated to its polar convex polyhedron $P^{\circ}$.
\end{rem}

In two dimensions, spherical mosaics may be characterized by the \emph{angle excess} associated with their cells which is equal to the solid angle subtended by the cell or, alternatively, the spherical area of the cell. Let $\mathcal{M}$ be a convex mosaic on $\mathbb{S}^2$ with $N_v$ nodes and $N_c$ cells. Since $\mathcal{M}$ is a tiling of $\mathbb{S}^2$, the average area of a cell is $\Omega_C = \frac{4\pi}{N_c}$. Similarly, the average area of a cell in the dual mosaic is $\Omega_N = \frac{4\pi}{N_v}$.

\begin{defn}\label{harmonicexcess}
For any spherical mosaic $\mathcal{M}$, we call the quantity $\bar{\mu}(\mathcal{M})= \frac{1}{\pi} \frac{\Omega_C \Omega_N}{\Omega_C + \Omega_N}$ the \emph{harmonic angle excess} of $\mathcal{M}$.
\end{defn}

\begin{prop}\label{prop:spherical_ftf}
The harmonic degree of any convex, face-to-face mosaic $\mathcal M$ on $\mathbb{S}^2$ is
\begin{equation}\label{eq:spherical}
\bar{h}(\mathcal{M}) = 2 - \bar{\mu}(\mathcal{M}).
\end{equation}
\end{prop}

\begin{proof}
Let $N_c$ and $N_v$ denote the numbers of cells and nodes of $\mathcal{M}$, and let $\bar{v}$ and $\bar{n}$ denote the average degree of a cell and a node, respectively. Then the number of adjacent pairs of cells and nodes of $\mathcal{M}$ is equal to
\begin{equation}\label{eq:spherical1}
\bar{v} N_c = \bar{n} N_v.
\end{equation}
Let $\alpha_{ij}$ denote the angle of the cell $C_i$ at the vertex $v_j$ if they are adjacent, and let $\alpha_{ij}=0$ otherwise.
We compute the sum $\sum_{i,j} \alpha_{ij}$ in two different ways.
First, note that $\sum_{i,j} \alpha_{ij}=\sum_j \sum_i \alpha_{ij} = 2 \pi N_v$.
On the other hand, the area of any cell $C_i$ is equal to the angle sufficit of $C_i$, or more specifically, $\area(C_i) = \sum_{j} \alpha_{ij} - (\deg(C_i) - 2) \pi$, where $\deg(C_i)$ is the number of vertices of $C_i$ (see Subsection~\ref{subsec:history}).
Since $\sum_{i} \area(C_i) = \area(\mathbb{S}^2) = 4\pi$ and $\sum_{i} \deg(C_i) = \bar{v} C$, it follows that $\sum_{i,j} \alpha_{ij} = 4 \pi + \bar{v} N_c \pi - 2N_c\pi$.
This implies the equality
\begin{equation}\label{eq:spherical2}
2 N_v = \bar{v} N_c - 2N_c + 4.
\end{equation}
Now, (\ref{eq:spherical}) follows from (\ref{eq:spherical1}), (\ref{eq:spherical2}) and the equation $\bar{\mu} = \frac{4}{N_v+N_c}$
(which follows from Definition \ref{harmonicexcess}).
\end{proof}

\begin{cor}\label{cor:spherical_h}
The harmonic degree of any face-to-face convex mosaic $\mathcal{M}$ of $\mathbb{S}^2$ is
\[
\bar{h}(\mathcal{M}) < 2.
\]
\end{cor}

While it does not seem feasible to extend the definition of $\bar{h}$ for mosaics in $\HH^2$ in a straightforward way, the geometric interpretation of this quantity in 
Subsection~\ref{qvary} permits us to find a variant of Corollary~\ref{cor:spherical_h} also in this case.

Let $\mathcal{M}$ be a convex face-to-face mosaic in any of the planes $\Re^2$, $\mathbb{S}^2$ or $\HH^2$.
Let $C$ be a cell of $\mathcal{M}$ with $v$ vertices. Let $p_j$, $j=1,2,\ldots,v$ be the vertices of $C$, and fix an arbitrary point $q \in \inter C$.
Let $L_j$ denote the sideline of $C$ passing through the vertices $p_j$ and $p_{j+1}$, and let $R_j$ denote the ray
starting at $q$ and intersecting $L_j$ in a right angle.
The convexity of $C$ implies that the rays $R_1, R_2, \ldots, R_v$ are in this cyclic order around $q$.
Let $\Omega_E(C,p_j)$ denote the angle of the angular region which is bounded by $R_{j-1} \cup R_j$ and whose interior is disjoint from all the rays $R_{j'}$.
Furthermore, let $\Omega_I(C,p_j)$ denote the interior angle of $C$ at $p_j$.
Now we define the quantity
\[
\bar{\Omega}(C) = \frac{\sum_{j=1}^v ( \Omega_E(C,p_j) +  \Omega_I(C,p_j)  )}{v} = \frac{2 \pi + \Lambda(C)}{v},
\]
where $\Lambda(C)$ is the sum of the interior angles of $C$.

Observe that if $\mathcal{M}$ is a Euclidean mosaic, then the weighted average value of $\bar{\Omega}(C)$, with the weight equal to $v$, over the family of all cells of $\mathcal{M}$ coincides with $\bar{\Omega}$.

Next, assume that $\mathcal{M}$ is a spherical mosaic. Let the cells of $\mathcal{M}$ be $C_i$, $i=1,2,\ldots, N_c$, and let $N_v$ and $N_e$ be the number of nodes and edges of the mosaic, respectively. If the degree of $C_i$ is $v_i$, then $\sum_{i=1}^{N_c} v_i = 2 N_e$, and by Euler's formula, $N_v + N_c = N_e + 2$, yielding $\bar{\mu} = \frac{4}{N_c + N_v} = \frac{4}{N_e+2}$.
Furthermore, for all values of $i$, the area formula for spherical polygons yields that $\Lambda(C_i) = (v_i-2) \pi + \area(C_i)$.
Thus, $\bar{\Omega}(C_i) = \frac{v_i \pi + \area(C_i)}{v_i}$. Since the total area of all cells is $4\pi$, this implies
\[
\bar{\Omega} = \frac{\sum_{i=1}^{N_c} \bar{\Omega}(C_i) v_i }{\sum_{i=1}^{N_c} v_i} = \frac{ \pi \sum_{i=1}^{N_c} v_i + \sum_{i=1}^{N_c} \area(C_i)}{2 N_e} = \frac{N_e \pi + 2\pi}{N_e}
= \frac{2\pi}{2 - \bar{\mu}}.
\]
We have shown that for face-to-face, convex mosaics on $\mathbb{S}^2$, we have $\bar{h} = 2 - \bar{\mu}$, (cf. (\ref{eq:spherical})). Thus, for these mosaics we have $\bar{h}= \frac{2\pi}{\bar{\Omega}}$, extending Theorem~\ref{thm:harmonic} for $2$-dimensional spherical mosaics.

Finally, consider the case that $\mathcal{M}$ is a hyperbolic mosaic. Let $C_i, i=1,2,\ldots$ denote the cells of $\mathcal{M}$, and let $v_i$ denote the degree of $C_i$. As in the spherical case, by the area formula for hyperbolic polygons, we have that $\bar{\Omega}(C_i) = \pi - \frac{\area(C_i)}{v_i} < \pi$ for all values of $i$.
For any nonnegative function $f : \mathbb{N} \to \Re$, we may define the \emph{harmonic degree of $\mathcal{M}$ with respect to $f$} as
\[
\bar{h}_f = \lim_{k \to \infty} \frac{2 \pi \sum_{i=1}^k f(i) v_i}{\sum_{i=1}^k f(i) v_i \bar{\Omega}(C_i) },
\]
where the inequalities $\bar{\Omega}(C_i) < \pi$, $i \in \mathbb{N}$ imply $\bar{h}_f \geq 2$.
Note that since any measure on a countable set is atomic, the above formula exhausts all reasonable possibilities for defining harmonic degree.

\subsection{Non face-to-face and nonconvex mosaics
}\label{ss:nff}

Conjecture \ref{con:1} formulates the hypothesis that the harmonic degree of  $d$-dimensional Euclidean face-to-face mosaics is confined to the range $(d, 2^{d-1}]$.  
In the current subsection we would like to point out that  in case of non face-to-face or nonconvex mosaics this range may be much broader.

First, we deal with non face-to-face mosaics.
According to the convention introduced in  Subsection \ref{subsec:history}, the degree of a node is equal to the number of vertices coinciding at that node,  both for face-to-face and non face-to-face convex mosaics.

In $d=2$ dimensions we already stated that for face-to-face mosaics we have $\bar h=2$ \cite[Theorem 10.1.6]{Schneider}, which is equivalent to
\begin{equation}\label{2D}
\bar n = \frac{2\bar v}{\bar v-2}.
\end{equation}
If we admit non face-to-face mosaics and we sum the internal angles over all cells and also sum the same angles as nodal angles over all nodes then (\ref{2D})
generalizes to
\begin{equation}\label{2D_1}
\bar n = \frac{2\bar v}{\bar v-p-1},
\end{equation}
where $p$ is the proportion of the regular nodes in the family of all nodes, where we call a node \emph{regular} if it is the vertex of every cell it belongs to. As we can see, in 2 dimensions convex mosaics
have two free parameters and they form a compact, 2D subset of the $[\bar n, \bar v]$ symbolic plane as illustrated in Figure \ref{fig:2}.
\begin{figure}[ht]
\begin{center}
\includegraphics[width=0.9\textwidth]{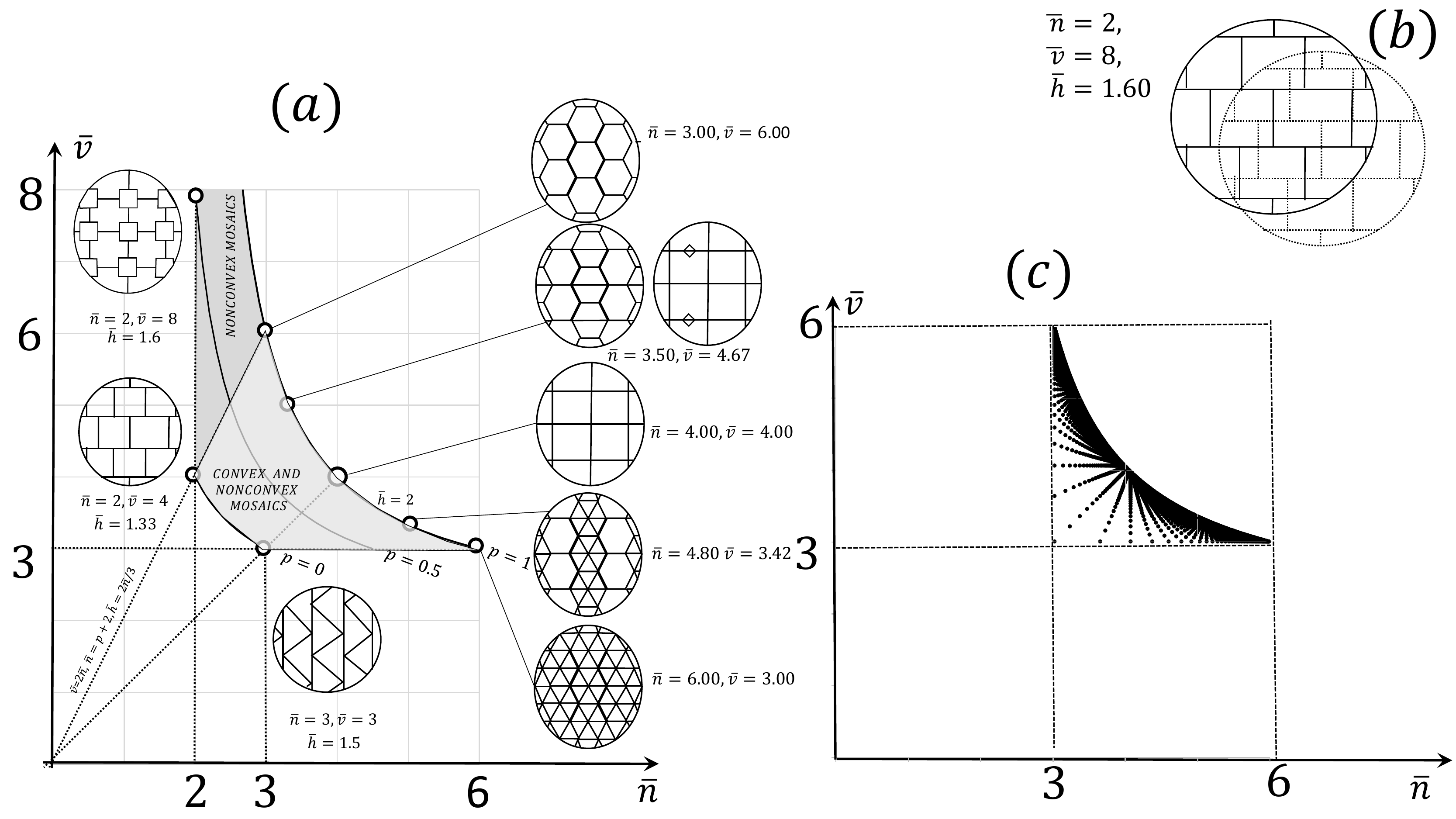}
\caption{(a) Symbolic plane for planar mosaics. The $p=1$ line corresponds to face-to-face mosaics. Grey shaded area marks the descriptors of all admissible mosaics
in the plane. (b) Example of a special 3D mosaic with $\bar h<2$. Solid line: odd layer, dotted line, even layer. Both layers correspond to the planar mosaic in panel (a) at $(\bar n, \bar v)=(2,4)$
. (c)  Parameter plane for spherical mosaics in $d=2$ dimensions. All mosaics shown with $N_c+N_v\leq 200$, $N_c$ denoting the number of cells, $N_v$ denoting the number of nodes. Mosaics on the $\bar v=3$ and $\bar n=3$ lines correspond to simple and simplicial polyhedra, respectively. Observe how mosaics accumulate on the line corresponding to face-to-face Euclidean mosaics.}
\label{fig:2}
\end{center}
\end{figure}
By computing the harmonic degree $\bar h$  over the admissible domain marked on Figure \ref{fig:2} (a) we find that $1.33 \leq \bar h \leq 2$
which indicates that non face-to-face mosaics may admit lower harmonic degrees than face-to-face mosaics.
Figure (\ref{fig:2}) (b) shows an example of a non face-to-face mosaic in $d=3$ constructed as alternated, shifted layers of a brick-wall-type
planar mosaic. At every node just 2 vertices meet so we have $\bar n=2$ and each cell is a cuboid yielding $\bar v=8$. This results in  a value $\bar h=1.6$ which is certainly below the maximal value of $\bar h=2$ for planar mosaics.

\begin{rem}
Using the proof of Proposition \ref{prop:spherical_ftf}, the generalization of formula (\ref{2D_1}) to 2D spherical mosaics is straightforward: 
\begin{equation}\label{eq:spherical_nff}
\bar{v} = \frac{(2-\bar \mu)\bar{n}}{\bar{n}+\bar \mu -1-p}.
\end{equation}
\end{rem}

We finish this subsection with the investigation of nonconvex mosaics in $\Re^d$. First, we note that the definitions and the notation introduced in Section~\ref{sec:intro} can be generalized in a natural way for mosaics in which each cell is a possibly nonconvex polytope homeomorphic to a ball.

Observe that in case of any convex, face-to-face mosaics in $\Re^d$, putting new nodes on existing facets and `breaking up' these facets by slighly moving the new nodes creates nodes of degree $2$. A careful application of this procedure yields a sequence of mosaics  in which the average nodal degree tends to $2$, and the average cell degree tends to $\infty$, resulting in a sequence of mosaics whose harmonic degree tends to $2$.

If we start with the regular cubic mosaic, we may apply the method of the proof of Theorem~\ref{thm:interval} to prove the following.

\begin{rem}
Let $d \geq 2$. For any $h^{\star} \in \left( 2, 2^{d-1} \right]$, there is a normal, face-to-face, possibly nonconvex mosaic $\mathcal{M}$ in $\Re^d$ with harmonic degree $\bar{h}(\mathcal{M}) = h^{\star}$.
\end{rem}

For the nonconvex case, we raise the following variant of Conjecture~\ref{con:1}.

\begin{conj}\label{con:nonconvex}
If $\mathcal{M}$ is a normal, face-to-face, possibly nonconvex mosaic in $\Re^d$, then $\bar{h}(\mathcal{M}) \in \left( 2, 2^{d-1} \right]$.
\end{conj}

\section{Summary}\label{s:sum}

In this paper we proposed to represent mosaics in the  $[\bar n, \bar v]$ symbolic plane of average nodal and cell degrees and we introduced the harmonic degree $\bar h$, constant values of which appear as curves in this space. We pointed out that these curves appear to have special significance: in $d=2$ dimensions all convex, face-to-face mosaics appear as points of the $h=2$ curve and a compact domain can be associated to non face-to-face mosaics.  We showed that in case of 2D spherical mosaics $\bar h$ differs only in a constant from the suitably averaged angle excess and this explains why points associated with 2D regular mosaics on manifolds with constant curvature are separated by the $\bar h=2$ line.

The most interesting geometric interpretation of  $\bar h$ appears to be  Theorem \ref{thm:harmonic}, stating that the harmonic degree is the inverse of the averaged sum of two angles associated with polar domains, one of which is the internal  angle of a cell at a vertex of the cell. We showed that this interpretation of $\bar h$ remains valid for Euclidean mosaics in arbitrary dimensions as well as 2D  spherical mosaics. The link established
in Theorem \ref{thm:harmonic} between the harmonic degree $\bar h$ and the average total angle $\bar \Omega$ illustrates that the combinatorial and metric properties of convex mosaics are closely related.
While $\Omega$ is constant in 2D (resulting in $\bar \Omega=\pi, \bar h=2$), in 3D 
there exists a broad range in which $\Omega$ may fluctuate. Nevertheless, we found that for a set of 60 mosaics (which included all uniform honeycombs as well as random mosaics) the actual fluctuation is very small and the harmonic degree of all investigated mosaics was in the range $3.3 \leq \bar h \leq 4$.
By using a recursive algorithm we also constructed $d$-dimensional Euclidean mosaics which approach, as the number of recursive steps tends to infinity,  the harmonic degree $\bar h = d$. We proved that $\bar h$ may assume any value in the interval $(d, 2^{d-1}]$. 

All the above computations and results led us to formulate Conjecture \ref{con:1}, stating that there exist no Euclidean, normal, face-to-face mosaics the harmonic degree of which lies outside the $(d, 2^{d-1}]$ interval. If true, this conjecture would not only yield an interesting alternative explanation for the averaged behavior of 1D and 2D mosaics but also deepen our current understanding of
3D (and higher dimensional) honeycombs.

\pagebreak 

\section*{Appendix}

\begin{table} [htp] 
{\tiny
\begin{tabular}{| c |r||c|c|c|r|}
\hline
ID.  & Name of mosaic & $\bar n$ & $\bar v$ & $\bar f$ & $\bar h$  \\
\hline
\hline
 1 & cubic    & 8 & 8 & 6 & 4.00		 \\
\hline
2 & rectified cubic	& 6 & 	9	& 11	& 3.60	 \\
\hline
3 & truncated cubic &	5	& 15	& 11	& 3.75	 \\
\hline
4 & cantellated cubic & 	5	& 12 & 	11.6 &	3.53	 \\
\hline
5 & cantitruncated cubic	& 4	& 19$\frac{1}{5}$&	11$\frac{3}{5}$ &	3.31	 \\
\hline
6 & runcitruncated cubic	& 5	 & 15	& 11	& 	3.75	\\
\hline
7 & alternated cubic	 &  14	 & 4$\frac{2}{3}$ &	5$\frac{1}{3}$	& 3.50		 \\
\hline
8 & cantic cubic	 & 5	& 15 & 	11	&  3.75		 \\
\hline
9 & runcic cubic	& 5	& 10	& 10	& 3.33	 \\
\hline
10	 & runcicantic cubic &	4	& 24	& 14	& 3.43	 \\
\hline
11	 & bitruncated cubic	 & 4	& 24	& 14	& 3.43		 \\
\hline
12	& omnitruncated cubic & 	4	& 24	 & 14	& 3.43	\\
\hline
13	& quarter cubic & 	8	& 8	& 6	& 4.00		 \\
\hline
14	& truncated/bitruncated square prismatic &	6	& 12	& 8	& 4.00	 \\
\hline
15	& snub square prismatic & 	10	& 6$\frac{2}{3}$ &	5$\frac{1}{3}$ &	4.00	 \\
\hline
16	& triangular prismatic	& 12	& 6	& 5	&  4.00	\\
\hline
17	& hexagonal prismatic	& 6	& 12	& 8	& 4.00		 \\
\hline
18	& trihexagonal prismatic &	8	& 8	& 6 & 	4.00	 \\
\hline
19	& truncated hexagonal prismatic	& 6	& 12 &	8	&  4.00	 \\
\hline
20	& rhombi-hexagonal prismatic	& 8	& 8	& 6	&  4.00		 \\
\hline
21		& snub-hexagonal prismatic	& 	10	& 	6$\frac{2}{3}$	& 	5$\frac{1}{3}$		& 4.00	 \\
\hline
22		& truncated trihexagonal prismatic		& 6		& 12		& 8		&  4.00	 \\
\hline
23		& elongated triangular prismatic		& 10		& 6$\frac{2}{3}$	& 5$\frac{1}{3}$	& 4.00		 \\
\hline
24		& gyrated alternated cubic		& 14		& 4$\frac{2}{3}$	& 5$\frac{1}{3}$		&  3.50	 \\
\hline
25		& gyroelongated alternated cubic		& 13		& 5$\frac{1}{5}$	& 5$\frac{1}{5}$		& 3.71	 \\
\hline
26		& elongated alternated cubic		& 13		& 5$\frac{1}{5}$		& 5$\frac{1}{5}$ &  3.71	 \\
\hline
27		& gyrated triangular prismatic		& 12		& 6		& 5		&  4.00	 \\
\hline
28		& gyroelongated triangular prismatic		& 10		& 6$\frac{2}{3}$		& 5$\frac{1}{3}$		& 4.00		 \\
\hline
29 & 	Poisson-Voronoi	& 4	& 27.07 &	15.51 &	 3.49	 \\
\hline
30	& Hyperplane &	8	& 8	& 6	&  4.00		 \\
\hline
 31 & iterated foam    & 4 & 12 & 8 & 3.00		 \\
\hline
1'	& dual of cubic	& 8	& 8	& 6	& 4.00		 \\
\hline
2'	& dual of rectified cubic &	9 &	6	& 8	&  3.60		 \\
\hline
3'	& dual of truncated cubic	& 15	& 5	& 5	&  3.75		 \\
\hline
4'	& dual of cantellated cubic	& 12	& 5	& 6	& 3.53		 \\
\hline
5'	& dual of cantitruncated cubic	& 19$\frac{1}{5}$	& 4 & 	4	&  3.31		 \\
\hline
6'	& dual of runcitruncated cubic	& 15	& 5 &	5	&  3.75	 \\
\hline
7'	& dual of alternated cubic	& 4$\frac{2}{3}$	& 14 &	12	&3.50		 \\
\hline
8'	& dual of cantic cubic &	15	& 5	& 5 &	  3.75	 \\
\hline
9' & 	dual of runcic cubic	& 10	& 5	& 6	& 3.33	 \\
\hline
10'	& dual of runcicantic cubic &	24	& 4	& 4	& 3.43	 \\
\hline
11'	& dual of bitruncated cubic	& 24	& 4	& 4	& 3.43		 \\
\hline
12'	& dual of omnitruncated cubic	& 24	& 4	& 4	& 3.43		 \\
\hline
13'	& dual of quarter cubic	& 8	& 8	& 6	& 4.00		 \\
\hline
14'	& dual of truncated/bitruncated square prismatic & 	12	& 6 & 5 &   4.00		 \\
\hline
15'	& dual of snub square prismatic	& 6$\frac{2}{3}$ & 	10 &	7	& 4.00	 \\
\hline
16'	& dual of triangular prismatic &	6 &	12	& 8	& 4.00		 \\
\hline
17'	& dual of hexagonal prismatic &	12	& 6	& 5 & 4.00	 \\
\hline
18'	& dual of trihexagonal prismatic &	8	& 8	& 6	& 4.00		 \\
\hline
19'	& dual of truncated hexagonal prismatic & 	12	& 6	& 5	& 4.00		 \\
\hline
20'	& dual of rhombi-hexagonal prismatic &	8	& 8	& 6	& 4.00		 \\
\hline
21'	& dual of snub-hexagonal prismatic & 	6$\frac{2}{3}$ &	10	& 7	& 4.00		 \\
\hline
22'	& dual of truncated trihexagonal prismatic &	12	& 6 & 5 &	 4.00		 \\
\hline
23'	& dual of elongated triangular prismatic &	6$\frac{2}{3}$	& 10	& 7	& 4.00		 \\
\hline
24'	& dual of gyrated alternated cubic &	4$\frac{2}{3}$ &	14 &	12 & 3.50		 \\
\hline
25'	& dual of gyroelongated alternated cubic &	5$\frac{1}{5}$ &	13 &	10	&  3.71	 \\
\hline
26'	& dual of elongated alternated cubic &	5$\frac{1}{5}$ &	13	& 10	& 3.71	 \\
\hline
27'	& dual of gyrated triangular prismatic & 	6	& 12 &	8 &	 4.00		 \\
\hline
28'	& dual of gyroelongated triangular prismatic	& 6$\frac{2}{3}$ &	10 &	7	& 4.00	 \\
\hline
29'	& Dual of Poisson-Voronoi: Poisson-Delaunay &	27.07 &	4	& 4	&3.49		 \\
\hline
30'	& Dual of Hyperplane: Hyperplane &	8 &	8	& 6	&  4.00		 \\
\hline
31' & dual of iterated foam	& 12 & 	4	& 4	& 3.00	 \\
\hline 
\end{tabular}
}
\normalsize
\vspace{0.5cm}
\caption{Harmonic degree $\bar h$ of uniform convex honeycombs, their duals, Poisson-Voronoi, Poisson-Delaunay and Hyperplane random mosaics, iterated foams and their duals.}
\label{tab:1}
\end{table}

\pagebreak


\begin{thebibliography}{99}

\bibitem{BLNP} K. Bezdek, Z. L\'angi, M. Nasz\'odi and P. Papez, \emph{Ball-polyhedra}, Discrete Comput. Geom. \textbf{38} (2007), 201-230.

\bibitem{Deza} M. Deza and M. Shtogrin, \emph{Uniform partitions of 3-space, their relatives and embedding}, European J. Combin. \textbf{21} (2000) , 807-814, doi: 10.1006/eujc.1999.0385


\bibitem{FTK} G. Fejes T\'oth and W. Kuperberg, \emph{Packing and covering with convex sets}, Chapter 3.3, pp. 799–860, Handbook of Convex Geometry, North-Holland, Amsterdam, 1993.

\bibitem{approx} S.P.Y. Fung, C. Wang and F.Y.L. Chin, \emph{Approximation algorithms for some optimal 2D and 3D triangulations}, In: Handbook of Approximation Algorithms and Metaheuristics, edited by T. F. Gonzalez, Chapman\& Hall/CRC, Boca Raton, FL, 2007.

\bibitem{Gaddum} J.W. Gaddum, \emph{The sums of the dihedral and trihedral angles in a tetrahedron}, Amer. Math. Monthly \textbf{59} (1952),
370-371.

\bibitem{GHSch} F. Gao, D. Hug and R. Schneider, \emph{Intrinsic volumes and polar sets in spherical space}, Math. Notes \textbf{41} (2003), 159–176.

\bibitem{Grunbaum}  B. Gr\"unbaum, \emph{Uniform tilings of 3-space}, Geombinatorics \textbf{4} (1994), 49–56.

\bibitem{Johnson} N. W. Johnson, \emph{Uniform Polytopes}, manuscript, (1991).

\bibitem{morgan} P. N. Chung, M. A. Fernandez, Y. Li, M. Mara, F. Morgan, I. R. Plata, N. Shah, L. S. Vieira, E. Wikner, \emph{Isoperimetric pentagonal tilings}, Notices Amer. Math. Soc. \textbf{59} (May, 2012), 632-640.

\bibitem{Senechal} M. ~ Senechal, \emph{Which tetrahedra fill space?}, Mathematics Magazine \textbf{54}(5) (1981), 227-243.

\bibitem{Schneider} R. ~Schneider and W.~ Weil, \emph{Stochastic and Integral Geometry}, Springer-Verlag, 2008.

\bibitem{Schulte1} E. Schulte, \emph{Symmetry of polytopes and polyhedra}, Handbook of discrete and computational geometry, 311–330,
CRC Press Ser. Discrete Math. Appl., CRC, Boca Raton, FL, 1997.

\bibitem{Schulte2} E. Schulte, \emph{Tilings}, Handbook of convex geometry, Vol. A, B, 899–932, North-Holland, Amsterdam, 1993. 

\bibitem{Steinitz1} E. Steinitz,\emph{\"Uber die Eulersche Polyderrelationen}, Arch. Math. Phys. \textbf{11} (1906), 86-88.

\bibitem{Steinitz2} E. Steinitz, \emph{Polyeder und Raumeinteilungen} , Enzykl. math . Wiss., Vol . 3 (Geometrie), Part 3AB12 (1922), 1-139.





\end{thebibliography}
\end{document}